\documentclass[11pt, reqno]{amsart}
\usepackage{color}
\usepackage{version}
\usepackage{amssymb}
\usepackage{mathtools}
\newtheorem{theorem}{Theorem}[section]
\newtheorem{lemma}[theorem]{Lemma}

\newtheorem{prop}[theorem]{Proposition}

\theoremstyle{definition}
\newtheorem{definition}[theorem]{Definition}

\newtheorem{sublemma}[theorem]{Sublemma}

\newcommand{\C}{\mathbb{C}}

\newcommand{\N}{\mathbb{N}}

\renewcommand{\P}{\mathbb{P}}

\numberwithin{equation}{section}

\begin{document}
\title[]{Squeezing functions and Cantor Sets.}
\author[L. Arosio]{L. Arosio$^{\dag}$}
\author{J. E. Forn\ae ss$^{\dag\dag}$}
\author{N. Shcherbina}
\author[E. F. Wold]{E. F. Wold$^{\dag\dag\dag}$}
\thanks{$^{\dag\dag}$ and $^{\dag\dag\dag}$ Supported by NRC grant no. 240569 }
\thanks{$^{\dag}$Supported by the SIR grant ``NEWHOLITE - New methods in holomorphic iteration'' no. RBSI14CFME}
\thanks{This work was done during the international research program "Several Complex Variables and Complex Dynamics"
at the Center for Advanced Study at the Academy of Science and Letters in Oslo during the academic year 2016/2017. }
\address{ L. Arosio: Dipartimento Di Matematica\\
Universit\`{a} di Roma \textquotedblleft Tor Vergata\textquotedblright\ \\
Via Della Ricerca Scientifica 1, 00133 \\
Roma, Italy} \email{arosio@mat.uniroma2.it}
\address{J. E. Forn\ae ss: Department of Mathematics, NTNU\\Norway}\email{john.fornass@ntnu.no}
\address{N. Shcherbina: Department of Mathematics\\ University of Wuppertal\\ 42097 Wuppertal\\ Germany}\email{shcherbina@math.uni-wuppertal.de}
\address{E. F. Wold: Department of Mathematics\\
University of Oslo\\
Postboks 1053 Blindern, NO-0316 Oslo, Norway}\email{erlendfw@math.uio.no}

%
%    General info
%
\date{\today}

\begin{abstract}
We construct ``large" Cantor sets whose complements resemble the unit disk arbitrarily well from the point of view of the squeezing function,
 and we construct ``large" Cantor sets whose complements do not resemble the unit disk from the point of view of the squeezing function.
Finally we show that complements of Cantor sets arising as Julia sets of quadratic polynomials have degenerate squeezing functions, despite of having 
Hausdorff dimension arbitrarily close to two.
\end{abstract}

\maketitle

\section{Introduction}

Recently there have been many studies of the boundary behaviour of the squeezing function (see Section 2 for the definition and references) in one and several complex variables.
In one complex variable there are two extremes: 
\begin{itemize}
\item[(1)] if $\gamma\subset b\Omega$ is an isolated boundary component of a domain $\Omega$ which is not a point, then 
\begin{equation}
\lim_{\Omega\ni z\rightarrow\gamma}S_\Omega(z)=1, 
\end{equation}
\item[(2)] if $\gamma\subset b\Omega$ is an isolated boundary component of a domain $\Omega$ which is a point, then 
\begin{equation}
\lim_{\Omega\ni z\rightarrow\gamma}S_\Omega(z)=0. 
\end{equation}
\end{itemize}
This suggests studying the boundary behaviour of $S_\Omega(z)$ where $\Omega=\mathbb P^1\setminus K$, and $K$ is a Cantor set.  
In \cite{AhlforsBeurling} Ahlfors and Beurling showed that there exist Cantor sets in $\P^1$ whose complement admits a bounded injective holomorphic function. In particular,  such complements admit a non-degenerate squeezing function, and so this class of domains is nontrivial from the point of view of the squeezing function.  

Our first result is the following.

\begin{theorem}\label{main}
For any $\epsilon>0$ there exists a Cantor set $Q\subset I^2$ with  2-dimensional Lebesgue measure greater than $1-\epsilon$, 
such that 
\begin{equation}
\lim_{\Omega\ni z\rightarrow  Q}S_{\Omega}(z)=1,
\end{equation}
and, moreover, $S_\Omega(z)\geq 1-\epsilon$ for all $z\in\Omega$, where $\Omega=\mathbb P^1\setminus  Q$.
\end{theorem}

We also show that there exist Cantor sets with completely different behaviour.
\begin{theorem}\label{thm2}
There exists a Cantor set $Q\subset\mathbb P^1$ such that  the following hold
\begin{itemize}
\item[(1)] for any point $x\in Q$ and any neighbourhood $U$ of $x$ we have that $U\cap Q$ has positive 2-dimensional Lebesgue measure, and
\item[(2)] $S_\Omega$ achieves any value between zero and one on $U\cap\Omega$, where $\Omega=\mathbb P^1\setminus Q$.
\end{itemize}
\end{theorem}

Finally we show that certain  Julia sets arising in one dimensional complex dynamics are Cantor sets which are degenerate from the point of view of the squeezing function, although they can have Hausdorff dimension arbitrarily close to two and thus their complements admit bounded holomorphic functions.   Recall that a compact set of Hausdorff dimension strictly larger than one has strictly 
positive analytic capacity, hence its complement admits bounded holomorphic functions (see e.g. \cite{Pajot},  part (B) of Theorem 64, page 74).

\begin{theorem}\label{thm3}
Let $f_c(z)=z^2 + c$ with $c\notin\mathcal M$. 
Then $\mathbb P^1\setminus\mathcal J_c$ does not admit a bounded injective holomorphic function.  
\end{theorem}

Here, $\mathcal J_c$ denotes the Julia set for the function $f_c$, and $\mathcal M$ denotes
the Mandelbrot set, so that $J_c$ is a Cantor set if and only if $c\notin\mathcal M$.

Other Cantor sets of this type were constructed by Ahlfors and Beurling \cite{AhlforsBeurling}.

%
%On the other hand, Ahlfors and Beurling showed that there exist linear Cantor sets of positive linear measure whose complements
%do not admit bounded injective holomorphic functions.   Here, we will show that there exist such cantor sets with Hausdorff dimension arbitrarily close to two.   Such examples arise in one dimensional complex dynamics.   
%
%\begin{theorem}
%Let $p:\mathbb C\rightarrow \mathbb C$ be a polynomial mapping, and assume that all the critical points of $p$
%are in the basin of attraction of infinity.   Then $\mathbb P^1\setminus\mathcal J$ does not admit a bounded injective holomorphic function.   
%\end{theorem}
%It is known that for any $\alpha<2$ there are $c\in\mathbb C\setminus\mathcal M$, such that $\mathcal J_c$
%has Hausdorff dimension at least $\alpha$, where $\mathcal J_c$ denotes the Julia set of the quadratic 
%map $f_c(z)=z^2+c$, and $\mathcal M$ denotes the Mandelbrot set.  For $\alpha>1$ the complements of these 
%Julia sets admit bounded holomorphic functions.  

\section{A ``large" Cantor Set whose complement resembles the unit disk - Proof of Theorem \ref{main}}

We give some definitions. Let $\triangle\subset \C$ denote the unit disc, and let $B_r(p)\subset \C$ denote the disk of radius $r$  centered at $p$.  Let  $\Omega\subset\mathbb P^1$ be a domain and let $x\in \Omega$. If  
 $\varphi:\Omega\rightarrow\triangle$  is an  injective holomorphic function such that $\varphi(x)=0$, we set
\begin{equation}
S_{\Omega,\varphi}(x):=\sup\{r>0:B_r(0)\subset\varphi(\Omega)\},
\end{equation}
and we set 
\begin{equation}
S_\Omega(x):=\sup_\varphi\{S_{\Omega,\varphi}(x)\},
\end{equation}
where the supremum is taken over  all injective holomorphic functions  $\varphi:\Omega\rightarrow\triangle$ such that $\varphi(x)=0$. The function $S_\Omega$ is called the {\sl squeezing function}. If the domain $\Omega$ does not admit any  bounded injective holomorphic functions, then the squeezing function is called degenerate.
The concept of squeezing function goes back to work by Liu-Sun-Yau, see \cite{LiuSunYau04} (2004),
\cite{LiuSunYau05} (2005) and S.-K- Yeung \cite{Yeung09} (2009).
More recently, Deng-Guan-Zhang, see \cite{DengGuanZhang12} (2012) initiated a basic study of the squeezing function. After that the squeezing function has been investigated by several authors, among them,
Forn\ae ss-Wold \cite{FornaessWold15} (2015), Nikolov-Trybula-Andreev \cite{NikolovTrybulaAndreev16} (2016), Deng-Guan-Zhang \cite{DengGuanZhang16}(2016), 
Joo-Kim \cite{JooKim16} (2016), Kim-Zhang\cite{KimZhang15} (2016), Zimmer \cite{Zimmer17} (2017),
Forn\ae ss-Rong \cite{FornaessRong17} (2017), Forn\ae ss-Shcherbina \cite{FornaessShcherbina17} (2017),
Diederich-Forn\ae ss \cite{DiederichFornaess15} and Forn\ae ss-Wold \cite{FornaessWold16}.
We  will introduce an auxiliary function $R$ that will enable us to bound the squeezing function from below on the limit 
of a certain increasing sequence of domains.  
Let $\Omega\subset\mathbb P^1$ be a domain which admits an injective holomorphic map  $\psi:\Omega\hookrightarrow\triangle$.  Then 
for any point $x\in\Omega$ it is known (see, for example, Theorem 1 in  \cite{ReichWarschawski}) that $\Omega$ also admits a circular slit map, that is an injective holomorphic map $\varphi:\Omega\hookrightarrow \triangle$ onto a circular slit 
domain $S$, such that $\varphi(x)=0$.   By definition, $S$ is a circular slit domain if $\triangle\setminus S$ consists of arcs lying on concentric 
circles centred at the origin (the arcs may degenerate to points).  If $x\in \Omega$, we let ${\sf Slit}_x(\Omega)$ denote 
the set of all circular slit maps  that sends $x$ to the origin.  For a domain $\Omega\subset\mathbb P^1$ we define 
\begin{equation}
R_\Omega(x):=\underset{\varphi\in\mathcal {\sf Slit}_x(\Omega)} {\sup} \{S_{\Omega,\varphi}(x)\}.
\end{equation}
%If $\Omega$ does not admit a bounded injective holomorphic function, we set $R_\Omega=0$ {\blue Why? Maybe it is better just to leave it undefined}.  
Notice that by definition $R_\Omega\leq S_\Omega$.

\begin{definition}
Let $\{\Omega_j\}_{j\in\mathbb N}$ be a sequence of domains in $\mathbb P^1$, and set $K_j:=\mathbb P^1\setminus\Omega_j$. 
We say that $\Omega_j$ converges \emph{strongly}
to a domain $\Omega\subset\mathbb P^1$ with $K:=\mathbb P^1\setminus\Omega$, if the compact sets $K_j$ converge to  $K$ in the Hausdorff distance, and we write $\Omega_j\overset{s}{\rightarrow}\Omega$.  If  $x_j\in\Omega_j$
and if $x_j\rightarrow x\in\Omega$ we write $(\Omega_j,x_j)\overset{s}{\rightarrow}(\Omega,x)$.
\end{definition}

%\begin{definition}
%Let $\{\Omega_j\}_{j\in\mathbb N}$ be a sequence of domains in $\mathbb P^1$, and let $\Omega\subset\mathbb P^1$ 
%be a domain.  We say that $\Omega_j\rightarrow\Omega$ with respect to kernel convergence if the following hold
%\begin{itemize}
%\item[(i)] for any compact set $K\subset\Omega$ there exists $N\in\mathbb N$ such that $K\subset\Omega_j$ for all $j\geq N$, and 
%\item[(ii)] for any open set $U\subset\mathbb P^1$ such that $U\subset\Omega_j$ for $j\geq N$ for some $N\in\mathbb N$, we 
%have that $U\subset\Omega$.
%\end{itemize}
%\end{definition}

\begin{prop}\label{cont}
Let $\Omega\subset\mathbb P^1$ be a finitely connected domain such that no boundary component of $\Omega$ is a point.  Let $N\in\mathbb N$, and let $\{\Omega_j\}$ be a sequence of domains, where each 
$\Omega_j$ is $m_j$-connected with $m_j\leq N$. 
Assume that $(\Omega_j,x_j)\overset{s}{\rightarrow}(\Omega,x)$.
Then $R_{\Omega_j}(x_j)\rightarrow R_\Omega(x)$.  
\end{prop}
\begin{proof}
Let $K_1,...,K_m$ denote the complementary components of $\Omega$.  Then for each $1\leq k\leq m$ there is a
\emph{unique} slit map $\varphi_k:\Omega\rightarrow\triangle$ such that $\varphi_k$ identifies   $K_k$ with $b\triangle$, $\varphi_k(x)=0$ and $\varphi_k'(x)>0$ (see, for example, Theorem 7 in \cite{ReichWarschawski}).
In particular, $R_\Omega(x)$ is realised by one (or more) of these maps.  \

Similarly, each $\Omega_j$ has complementary components $K^j_{k}$ for $1\leq k\leq m_j$,  and for the $K^j_k$'s that are not points, there are unique slit maps 
$\varphi^j_k$  identifying $K^j_k$ with $b\triangle$,  $\varphi^j_k(x_j)=0$ and  $(\varphi^j_k)'(x_j)>0$.  \

After re-grouping to simplify notation, we may assume that there is a sequence $s_j\leq m_j$  such that the compact set $K^j_1\cup\cdot\cdot\cdot\cup K^j_{s_j} $ converges in the Hausdorff distance to a complementary component of $\Omega$, say $ K_1$,
and groups of the other $K^j_i$'s converge to the other complementary components of $\Omega$.  Since the 
diameter of $K_1$ is strictly positive, we may assume that there is a lower bound for the diameters of the sets $\{K^j_1\}$.

We first claim that there exists a constant $c>0$ such that $(\varphi^j_{1})'(x_j)>c$ for all $j$. Notice that, in view of Koebe's $\frac{1}{4}$-theorem, our claim implies that  all slits are bounded away from zero. 
 Assume 
by contradiction that there is such a sequence  $(\varphi^j_{1})'(x_j)\rightarrow 0$. 
For any convergent subsequence of the sequence $(\varphi^j_{1})$, the limit map is constantly equal to $0$. Choose such a convergent subsequence and denote it still by $(\varphi^j_{1})$.  
After possibly having to pass to a subsequence, we may now choose a nontrivial loop $\tilde\gamma:=bB_r(0)$, where $0<r<1$, which is  contained in $\varphi^j_1(\Omega_j)$ for all $j$, such that the Kobayashi length of $\tilde\gamma$ in $\varphi^j_1(\Omega_j)$ is uniformly bounded from above.

 Let $U$ be any (small) open neighbourhood of $K_1$.  Then for $j$ sufficiently 
large, we have that $\varphi^j_1(\Omega_j\setminus U)$ is contained in the disk bounded by $\tilde\gamma$.
Set $\tilde\gamma^j:=(\varphi^{j}_1)^{-1}(\tilde\gamma)$.   Then since $\varphi^j_1$ identifies $K^j_1$ with 
$b\triangle$, we have that $\Omega_j\setminus U$ is on one side of $\tilde\gamma_j$ and $K^j_1$ 
is on the other.  Then the spherical lengths of the $\tilde\gamma_j$'s are bounded uniformly from 
below, since the diameters of the $K^j_1$'s are bounded uniformly from below.   But then the Kobayashi length  of $\tilde\gamma^j$ in $\Omega_j$ goes to infinity, a contradiction. 
So we 
may extract a subsequence from $\varphi^j_{1}$ converging uniformly on compact subsets of $\Omega$ 
to an injective map $\tilde\varphi:\Omega\rightarrow\triangle$.

We claim  that $\tilde\varphi$ maps 
$\Omega$ onto a slit domain. First we show that the slits cannot close up to a circle of radius strictly less than one.
Indeed, fix a compact set $L$ in $\Omega.$ Then we can assume that the sequence converges uniformly on $L$.
This implies that if there is a slit $S$ of minimal radius $r<1$ which closes up to a circle in the limit, then
eventually all the images of $L$ must be contained in the disc of radius $r.$ Then arguing as above we can pick a circle of radius $r<s<1$ so that on the preimages of the circle, the Kobayashi length is arbitrarily large. This is impossible.

We can assume that the slits converge. The complement of the limiting slits is connected. Pick any compact subset $F$
of the complement of the limiting slits. Consider the inverse maps of the $\varphi^j_k$. 
This is a normal family. Indeed,  we can remove a small disc around  a point  where the  sequence $\varphi^j_k$ is uniformly convergent. 
After this removal the family of inverses is a normal family.
The limit map of the inverses is then the inverse of a slit map on $\Omega$, which proves that the limit map $\tilde \varphi$ is a slit map. \

So $\varphi$ is, up to rotation, the unique slit map which identifies $K_1$ with $b\triangle$, and 
by choosing other complementary components than $K_1$ in the above construction, 
all the possible slit maps $\varphi_k:\Omega\rightarrow\triangle$
may be obtained as such limits.  So we would 
arrive at a contradiction if we did not have $R_{\Omega_j}(x_j)\rightarrow R_\Omega(x)$.
\end{proof}

\begin{lemma}\label{step}
Let $Q_j=[a_j,b_j] \times [c_j,d_j]\subset\mathbb C$ be pairwise disjoint cubes for $j=1,...,m$, and set 
\begin{equation}
\Omega:=\mathbb P^1\setminus (Q_1\cup\cdot\cdot\cdot\cup Q_m).
\end{equation}
For each $j$, set 
\begin{equation}
\Gamma_j:=\{a_j+(1/2)(b_j-a_j)\}\times [c_j,d_j], 
\end{equation}
and for $k\in\mathbb N$ denote by $\Gamma_j(1/k)$ the open $\frac{1}{k}$-neighborhood of $\Gamma_j$, and by 
$Q_{j,k}^l$ and  $Q_{j,k}^r$ the left and right connected components of $Q_j\setminus\Gamma_j(1/k)$ respectively.
Set 
\begin{equation}
\Omega_k:=\mathbb P^1\setminus (Q_{1,k}^l\cup Q_{1,k}^r\cup\cdot\cdot\cdot\cup Q_{m,k}^l\cup Q_{m,k}^r).
\end{equation}
Then for any $\epsilon>0$ there exist $\delta>0$ and $N\in\mathbb N$ such that, for all $k\geq N$,
\begin{equation}\label{smallnear}
R_{\Omega_k}(z)\geq 1-\epsilon \mbox { if  }z\in\Omega_k\cap Q_j(\delta) \mbox{ for some } j,
\end{equation}
and 
\begin{equation}\label{approx}
|R_{\Omega_k}(z)-R_{\Omega}(z)|<\epsilon \mbox{ if } z\notin Q_j(\delta) \mbox{ for all } j,
\end{equation}
where $Q_j(\delta)$ denotes the $\delta$-neighborhood of $Q_j$.

\end{lemma}

\begin{proof}
Since all cases are similar, to avoid notation, we prove \eqref{smallnear} for $j=1$. 
We may also assume that $Q_1=[-1,1]\cup [-1,1]$.  \

For each $k\in \N$ there is a unique conformal map $\phi_k: \mathbb{P}^1\setminus   Q_{1,k}^l\cup Q_{1,k}^r  \rightarrow\mathbb P^1$ such that the image is the complement of two closed disks $B_k^1$ and $B_k^2$, normalized 
by the condition
\begin{equation}
\phi_k(z)= z + \sum_{j=1}^\infty a^k_j(1/z)^j
\end{equation}
near infinity  (see e.g. \cite{Goluzin}, Theorem 2, page 237). 
Then by uniqueness, $\phi_k(z)=-\phi_k(-z)$, so the two disks have the same size. 
Moreover, since each $\phi_k$ is  normalized 
to have derivative one at infinity, the radii of the disks have to be bounded from above and from below:
we can assume that the centers and the radii (in the spherical metric) converge.  Indeed, by the Koebe 1/4-theorem the discs must all be in a bounded region of $\mathbb C.$ Hence the radii are bounded above.
Next we assume that the radii converge to $0.$
Let $p,q$ denote the limits of the centers. Then the inverses are a normal family in the complement of the two points, hence the limit must be constant. This is only possible if
$p,q=\infty$ contradicting the uniform boundedness of the discs. 
%{\blue Q: what I do not understand here is why we need the previous argument (from `` Moreover" to ``contradicting the uniform boundedness of the discs'').
%It seems to me that if the image of $\phi_k$ is the complement of two disks of the same size (symmetric with respect to the origin) we can always rotate and rescale each map $\phi_k$ in such a way to obtain $B_k^1=B_1(-1-\delta_k) \mbox{ and } B_k^2=B_1(1+\delta_k)$, and then we can show as below that $\delta_k$ goes to 0.}
%
%{\color{cyan} I think we can also argue like that.  But we should then add an extra argument why we get uniform bounds on the derivatives at infinity, so that we can take limits of $\{\phi_j\}$ below.  But that is Schwarz Lemma.}

So by scaling 
and rotation, we may then assume that 
\begin{equation}
B_k^1=B_1(-1-\delta_k) \mbox{ and } B_k^2=B_1(1+\delta_k),
\end{equation}
for some $\delta_k > 0$,  where in general $B_r(p)$ denotes the disk of radius $r$ centered at $p$ (however, we have now possibly destroyed the normalization condition).

We now show that  $\delta_k\rightarrow 0$. Otherwise, consider the circles $|z-(1+\delta_k)|=1+\delta_k.$ These have uniformly bounded Kobayashi length. However their preimage goes around one of the rectangles and passes between the rectangles, where the Kobayashi metric is arbitrarily large. Hence their Kobayashi length is unbounded, contradiction.

 Since we may assume that the sequence $\{\phi_k\}$ converges to a conformal map, we get \eqref{smallnear}  from Lemma \ref{uniform} below. And since all cases are similar, we conclude 
that  \eqref{smallnear} holds for any $j=1,...,m$. 
Finally \eqref{approx} follows from Proposition \ref{cont}.
\end{proof}

\begin{lemma}\label{uniform}
Set $\Omega:=\mathbb P^1\setminus (\overline B_1(-1)\cup\overline B_1(1)\cup K)$ be a domain, with 
$K$ a compact set with finitely many connected components, disjoint from $\overline B_1(-1)\cup\overline B_1(1)$.
Let $\delta_j\searrow 0$, and suppose that 
\begin{equation}
\Omega_j:=\mathbb P^1\setminus (\overline B_1(-1-\delta_j)\cup \overline B_1(1+\delta_j)\cup K_j)
\end{equation}
is a sequence of domains such that $K_j\rightarrow K$ with respect to Hausdorff distance, and 
such that the number of connected components of $K_j$ is uniformly bounded.   
Then for any $\epsilon>0$ there exists $\eta>1$ such that $R_{\Omega_j}(z)\geq 1-\epsilon$
for all $j$ large enough such that $\overline B_1(-1-\delta_j)\subset B_\eta(-1)$ and $\overline B_1(1+\delta_j)\subset B_\eta(1)$,
and for all $z\in (B_\eta(-1)\cup B_\eta(1))\cap\Omega_j$.
\end{lemma}

\begin{proof}
Assume to get a contradiction that there exist $\epsilon>0$ and  sequences $\eta_k\searrow 1$, $j_k\rightarrow\infty$, such that
$$\overline B_1(-1-\delta_{j_k})\subset B_{\eta_k}(-1),\quad \overline B_1(1+\delta_{j_k})\subset B_{\eta_k}(1)$$ and  a sequence $z_k\in(B_{\eta_k}(-1)\cup B_{\eta_k}(1))\cap \Omega_{j_k}$  such that $R_{\Omega_{j_k}}(z_k)< 1-\epsilon$.  We may assume that $\mathrm{Re}(z_k)\geq 0$ for all $k$.\

Set $f_k(z):=z-(1+\delta_{j_k})$, $\Omega_{j_k}':=f_{k}(\Omega_{j_k})$, and $z_{k}':=f_{k}(z_k)$.
Note that  $1<|z_k'|\leq 2\eta_k-1$  and that $f_k(-2-\delta_{j_k})=-3(1+(2/3)\delta_{j_k})$.  Next set 
$g_k(z):=1/z$, $\Omega_{j_k}'':=g_k(\Omega_{j_k}')$, and $z_k'':=g_k(z_k')$.   
Then $|z_k''|\geq\frac{1}{2\eta_k-1}$ and $|g_k(f_k(-2-\delta_{j_k}))|<1/3$.  \

To sum up: $\Omega_{j_k}''$ is a domain obtained by removing a disk $D_k$  and the compact set $g_k(f_k(K_{j_k}))$ from the unit disk, 
the point $q_k$  on the boundary of $D_k$ closest to the origin is of modulus less than one third, and there 
is a point $z_k''\in\Omega_{j_k}''$ with $|z_k''|\geq\frac{1}{2\eta_k-1}$
for which $R_{\Omega_{j_k}''}(z_k'')< 1- \epsilon$.  

Clearly, the Poincar\'{e} distances
between $z_k''$ and $q_k$, and $z_k''$ and $g_k(f_k(K_{j_k}))$, goes to infinity as $k\rightarrow\infty$, so if we set
$\psi_k(z):=\frac{z-z_k''}{1-\overline z_k'' z}$, after possibly having to pass to a subsequence and in view of the following below sublemma, the domains $\psi_k(\Omega_{j_k}'')$
converge to a simply connected domain with respect to strong convergence.
Applying Proposition \ref{cont} and using one more time the following sublemma, this implies that $R_{\psi_k(\Omega_{j_k}'')}(0)\rightarrow 1$ as 
$k\rightarrow\infty$ - a contradiction.
\end{proof}

\begin{sublemma}
We have that $\liminf_{k\rightarrow\infty}d_P(z''_k,D_k)>0$, where $d_P$ denotes the Poincar\'{e} distance. 
\end{sublemma}

\begin{proof}
Note that $\mathrm{Re}(z_k')\geq -1- \delta_{j_k}$, so that if we set $\gamma_k:=\{z\in\mathbb C:\mathrm{Re}(z)=-1-\delta_{j_k}\}$, 
then any curve connecting $z_k''$ and $D_k$ will have to pass through $\tilde\gamma_k=g_k(\gamma_k)$.  So 
it is enough to find a lower bound for the Poincar\'{e} distance between $D_k$ and $\tilde\gamma_k$ for large $k$.
Now the real points on $\tilde\gamma_k$ are $0$ and $\frac{1}{-1-\delta_{j_k}}$, and the real points on $bD_k$
are $\frac{1}{-1-2\delta_{j_k}}$ and $\frac{1}{-3(1+(2/3)\delta_{j_k})}$, and using the fact that Poincar\'{e} disks are Euclidean disks, 
it suffices to control the distance between $\frac{1}{-1-\delta_{j_k}}$ and $\frac{1}{-1-2\delta_{j_k}}$, and between $0$ and $\frac{1}{-3(1+(2/3)\delta_{j_k})}$.
The last distance is clearly bounded away from zero, so we compute the first.   We have that 
$$
\lim_{k\rightarrow\infty }\log\frac{1+\frac{1}{1+\delta_{j_k}}}{1-\frac{1}{1+\delta_{j_k}}} - \log\frac{1+\frac{1}{1+2\delta_{j_k}}}{1-\frac{1}{1+2\delta_{j_k}}}  = \lim_{k\rightarrow\infty} \log \frac{1 - \frac{1}{1+2\delta_{j_k}}}{1 - \frac{1}{1+\delta_{j_k}}}= \log 2.
$$

\end{proof}

\begin{lemma}\label{tentative}
Let $\Omega\subset\mathbb P^1, x\in\Omega$, and suppose 
that $\mathbb P^1\setminus\Omega$ contains at least three points. 
Suppose that
$\Omega_j\overset{s}\rightarrow \Omega$.   
Then $$r:=\limsup_{j\rightarrow\infty}S_{\Omega_j}(x)\leq S_{\Omega}(x).$$ 
\end{lemma}

\begin{proof}
If $r=0$ this is clear, so we assume that $r>0$.   Then, after possibly having to pass to a subsequence, there exists a sequence $\varphi_j: \Omega_j\rightarrow\triangle$
of embeddings, $\varphi_j(x)=0$, $ B_{r_j}(0)\subset\varphi_j(\Omega_j)$, $r_j\rightarrow r$.  Let $a_j\in\mathbb P^1$ be distinct 
points such that $a_i\notin\Omega$ for $i=1,2,3$.  For any $\delta>0$ the ball $B_\delta(a_i)$ is not contained in $\Omega_j$ for 
all $j$ large enough.  So we may fix $0<\delta<<1$, and assume that 
there exist points $a^j_i\in B_\delta(a_i)$ such that $a^j_i\notin\Omega_j$
for all $j$ and for $i=1,2,3$.
Since there is a compact family of M\"{o}bius transformations mapping 
the triples $\{a^j_1,a^j_2,a^j_3\}$ to the triple $\{a_1,a_2,a_2\}$, and 
since the complement of three points is Kobayashi hyperbolic, we may assume   that for all $0<r'<r$ the sequence $\varphi_j^{-1}|_{B_{r'}(0)}$ is convergent.
Hence the derivatives of $\varphi_j'(x)$ are uniformly bounded below and above. Therefore we can assume that
the $\varphi_j$ converge to an injective holomorphic map from $\Omega$ to $\triangle$. Moreover the image contains the disc of radius $r.$

\end{proof}

\emph{Proof of Theorem \ref{main}:}

Set $Q^1=I^2$.   By alternating  Lemma \ref{step}  and its horizontal analogue we obtain a 
decreasing sequence $Q^j$ of disjoint unions of cubes, such that 
\begin{itemize}
\item[(1)] $Q:=\cap_{j\geq 1} Q^j$ is a Cantor set with 2-dimensional Lebesgue measure arbitrarily close to one, 
\item[(2)] $R_{\mathbb P^1\setminus Q^j}\geq 1 - \epsilon$, and 
\item[(3)] For any sequence $(z_j)$ in $\mathbb P^1\setminus Q$ converging to  $Q$ and any $\delta>0$ there exists an $N\in\mathbb N$ such that
$R_{\mathbb P^1\setminus Q^i}(z_j)>1-\delta$ whenever $j\geq N$ and $i$ is large enough (depending on $j$).  
\end{itemize}
%Let  $(z_j)$ be a sequence in  $\mathbb P^1\setminus Q$ converging to $Q$, and let $\delta>0$. By (3), for any $j\geq N$ there exists a sequence of slit maps $\varphi_i:\mathbb P^1\setminus Q^i\rightarrow\triangle$ satisfying $\varphi_i(z_j)=0$ 
%and $\varphi_i(\mathbb P^1\setminus Q^i)\supset{\blue B_{1-\delta}(0)}$.   By Koebe's 1/4-theorem all derivatives
%at $z_j$ are bounded uniformly away from zero, and so we may assume that the sequence  $(\varphi_i)$ converges uniformly on compact subsets    to a univalent map $ \varphi:\mathbb P^1\setminus Q\rightarrow\triangle$ with $\varphi(z_j)=0$.   Since we may also assume that $\varphi_i^{-1}|_{B_{1-\delta}(0)}$ is convergent,  we get that 
%$B_{1-\delta}(0)\subset\varphi(\mathbb P^1\setminus Q)$, that is $S_{\mathbb P^1\setminus Q}(z_j)\geq 1-\delta$.
%{\blue By (3) and by Lemma \ref{tentative} it follows that $S_{\mathbb P^1\setminus Q}(z_j)\geq 1-\delta$ for all $j\geq N$.}
By (1) the two-dimensional Lebesgue measure of $Q$ can be arbitrarily close to one.
By (2) and by    Lemma \ref{tentative}  it follows that $S_{\mathbb P^1\setminus Q}(z)\geq 1-\epsilon$. 
By (3) and by Lemma \ref{tentative} it follows that $S_{\mathbb P^1\setminus Q}(z_j)\geq 1-\delta$ for all $j\geq N$, and  hence $\lim_{\mathbb P^1\setminus Q\ni
z\rightarrow  Q}S_{\mathbb P^1\setminus Q}(z)=1$.
%For any point $z\in\mathbb P^1\setminus Q$ we can consider $S_{\mathbb P^1\setminus Q}(z)\geq 1-\epsilon$
%a sequence of slit maps $\phi_j$ realising the values $R_{\mathbb P^1\setminus Q^j}(z)$, 
%and we can pass to a convergent subsequence.  The limit is a candidate for a map 
%realising the squeezing function at $z$, and so by Lemma \ref{tentative} we get $S_{\mathbb P^1\setminus Q}(z)\geq 1-\epsilon$. 
%A similar argument using (3) gives that $\lim_{\Omega\ni z\rightarrow  Q}S_{\Omega}(z)=1$. 
$\hfill\square$

\section{A ``large" Cantor Set whose complement does not resemble the unit disk - Proof of Theorem \ref{thm2}}

We modify the construction in the previous section.  For an inductive construction, assume that we have constructed 
a family $\mathcal{Q}^j:=\{Q_1^j,\dots, Q_{m(j)}^j\}$ of $m(j)$ disjoint cubes.  We may choose $m(j)$ closed loops $\Gamma^j_i$, each surrounding 
and being so close to one of the cubes, that $S_{\Omega_j}(z)\geq 1-1/j$ if $z\in\Gamma^j_i$ for 
some $1\leq i \leq m(j)$, where $\Omega_j=\mathbb P^1\setminus \cup_iQ_i^j$.  Further we may choose a finite number of points $p^j_1,...,p^j_{k(j)}$ in $\Omega_j$ such that we find a point 
$p^j_\ell$ in any $1/j$-neighbourhood of any point in $b\Omega_j$, and such that 
$S_{\Omega'_j}(z)\geq 1-2/j$ if  $z\in \cup_{1\leq i\leq m(j)}\Gamma^j_i$, where we denote  $\Omega'_j:=\Omega_j\setminus \cup_{1\leq \ell\leq k(j)}\{p^j_\ell\}$.  The reason is that the removal 
of a set sufficiently close to the boundary of a domain, will essentially not disturb a lower bound for 
neither $S$ nor $R$.

Then by  Lemma \ref{smallcubes}  below and  Proposition \ref{cont} we may choose an arbitrarily small $\delta_j>0$
and arbitrarily small cubes $\tilde Q^j_\ell\subset \{|z-p^j_\ell|<\delta_j\}$ such that 
\begin{itemize}
\item[(i)] $S_{\Omega''_j}(z)\leq 1/j$ if $|z-p^j_\ell|=\delta_j$ for some $1\leq \ell\leq k(j)$, and 
\item[(ii)] $S_{\Omega''_j}(z)\geq 1 - 3/j$ if $z\in\Gamma^j_i$ for some $1\leq i\leq m(j)$,
\end{itemize}
where we denote $\Omega_j'':=\Omega_j\setminus \cup_\ell \tilde Q^j_\ell$.
By applying Lemma \ref{step} twice we may divide each cube in the collection 
$$ 
\{Q_1^j,\dots, Q_{m(j)}^j,   \tilde Q^j_1,\dots, \tilde Q^j_{k(j)}   \}
$$ into 
four, creating a new collection of cubes $\mathcal{Q}^{j+1}$ such that 
\begin{itemize}
\item[(i)] $S_{\Omega_{j+1}}(z)\leq 2/j$ if $|z-p^j_\ell|=\delta_j$ for some $1\leq \ell\leq k(j)$, and 
\item[(ii)] $S_{\Omega_{j+1}}(z)\geq 1 - 4/j$ for $z\in\Gamma^j_i$ for some $1\leq i\leq m(j)$,
\end{itemize}
where $\Omega_{j+1}$ denotes the complement of the cubes in $\mathcal{Q}^{j+1}$.
The inductive step may be repeated indefinitely so as to ensure that for all $k>1$ we still have that 
\begin{itemize}
\item[(i')] $S_{\Omega_{j+k}}(z)< 3/j$ for $|z-p^j_\ell|=\delta_j$ for some $\ell$, and 
\item[(ii')] $S_{\Omega_{j+k}}(z)>1 - 5/j$ for $z\in\Gamma^j_i$ for some $i$.  
\end{itemize}
 We now define
$$
Q \coloneqq \limsup_{j \rightarrow \infty} {\bigcup_{Q_l^j \in \mathcal{Q}^j}{Q_l^j}}
$$
\noindent
If in each step of the construction the points $p^j_i$ were 
chosen close enough to each of the previously constructed cubes, it follows that any 
connected component of $Q$ must be a point (since the diameters of the cubes go to zero), and 
no point will be isolated.  Hence $Q$ is a Cantor set.  It follows from Lemma \ref{smallcubes} that 
we may arrange that statement corresponding to (i') holds in the limit.   The statement corresponding to (ii') holds in the limit by Lemma \ref{tentative} since  $\Omega_j$ converges strongly to $\P^1\setminus Q$.

\begin{lemma}\label{points}
Let $K\subset\mathbb P^1$ be a compact set such that $\Omega=\mathbb P^1\setminus K$
admits a bounded injective holomorphic function.   Let $p_1,...,p_m\in\Omega$ be distinct points, 
and set $\Omega'=\Omega\setminus\{p_1,...,p_m\}$.  Then 
\begin{equation}
\lim_{\Omega'\ni z_j\rightarrow p_j}S_{\Omega'}(z)=0, 
\end{equation}
for $j=1,...,m$.
\end{lemma}
\begin{proof}
We consider $p_1$.  Assume to get a contradiction that there exists a sequence 
$\Omega'\ni z_j\rightarrow p_1$ and injective holomorphic maps $\varphi_j:\Omega'\rightarrow\triangle, \varphi_j(z_j)=0$, 
and $B_r(0)\subset\varphi_j(\Omega')$ for some $r>0$.  All maps extend holomorphically across $p_1,....,p_m$, 
and we may extract a subsequence converging to a limit map $\varphi$, with $\varphi(p_1)=0$.  Since 
$|\varphi(p_1)-\varphi(z_j)|\rightarrow 0$ as $j\rightarrow\infty$, this leads to a contradiction.  
\end{proof}

%
%\begin{lemma}\label{smallcubes}
%Let $K\subset\mathbb P^1$ be a compact set such that $\Omega=\mathbb P^1\setminus K$
%admits a bounded injective holomorphic function.   Let $p_1,...,p_m\in\Omega$ be distinct points, 
%and set $\Omega'=\Omega\setminus\{p_1,...,p_m\}$.   Fix (small) $\epsilon,\delta>0$ such that 
%$S_{\Omega'}(z)<\epsilon$ for all $z\in\Omega'$ with $|z-p_j|=\delta$ for some $j$.  Then 
%for sufficiently small cubes $Q_1,...,Q_m$ centred at $p_1,....,p_m$ and $\Omega''=\Omega\setminus(Q_1\cup\cdot\cdot\cdot\cup Q_m)$, 
%we have that  $S_{\Omega''}(z)<2\epsilon$ for $|z-p_j|=\delta$.
%\end{lemma}
%This{\red, in view of Lemma \ref{points},easily} follows {\red by contadiction after passing to the limit for injective holomorphic maps realizing the squeezing function $S_{\Omega''}(z)$ for points $z$ with $|z-p_j|=\delta$ as sizes of the cubes $Q_1,...,Q_m$ tend to zero.}

\begin{lemma}\label{smallcubes}
Let $K\subset\mathbb P^1$ be a compact set such that $\Omega=\mathbb P^1\setminus K$
admits a bounded injective holomorphic function.  Let $p_1,...,p_m\in\Omega$ be 
distinct points, and let $\epsilon>0$.  Then there exist $\delta_1>0$ (arbitrarily small)
and $0<\delta_2<<\delta_1$, 
such that for any 
domain $\Lambda\subset\mathbb P^1$ with $\mathbb P^1\setminus \Lambda\subset K(\delta_2)\cup(\cup_{j=1}^m B_{\delta_2}(p_j))$ (with at least one complementary component in each $B_{\delta_2}(p_j)$), we have that $S_{\Lambda}(z)<\epsilon$
for all $z\in\Lambda$ with $|z-p_j|=\delta_1$ for some $j$.  Here $K(\delta_2)$ denotes the $\delta_2$-neighbourhood of $K$. 
\end{lemma}
\begin{proof}
Let $0<\mu<<1$ (to be determined).   Fix $\delta_1>0$ such that the Kobayashi length  in $\Omega'=\Omega\setminus\{p_1,...,p_m\}$
of each loop $|z-p_j|=\delta_1$ is strictly less than $\mu$.  Let $f_\theta:\triangle\rightarrow\Omega'$ be a continuous family of universal covering maps with $f_\theta(0)=p_1+\delta_1e^{i\theta}$.
Then the Kobayashi metric $g^{\Omega'}_K(p_1 + \delta_1e^{i\theta})$ is equal to $1/|f_\theta'(0)|$.
Fix any $0<r<1$.   Then for any domain $\Omega''\subset\mathbb P^1$ that covers 
the union $\cup_\theta f_\theta(\overline{\triangle_r)}$, which is a compact subset of $\Omega'$, 
we have that $g^{\Omega''}_K(p_1 + \delta_1e^{i\theta})$ is bounded 
from above by $1/|r\cdot f_\theta'(0)|$.  So for $r$ sufficiently close to $1$ the Kobayashi 
length of the loop $|z-p_1|=\delta_1$ in $\Omega''$ is less than $\mu$ for any such domain.  
The same argument may be applied to all points $p_j$.  \

Now for any such domain $\Omega''$ we estimate the squeezing function 
with respect to $\mu$.   Write $S_{\Omega''}(p_j+\delta_1e^{i\theta})=s$, let 
$g:\Omega''\rightarrow\triangle$ be a map that realises the squeezing function 
at $p_j+\delta_1e^{i\theta}$, and let $\Gamma_j$ denote the loop $|z-p_j|=\delta_1$.
Then, since $g(\Gamma_j)$ is a nontrivial loop in $g(\Omega'')$ we have that 
$l_K(\Gamma_j)\geq\log (\frac{1+s}{1-s})$.  Then $s\leq\frac{e^\mu - 1}{e^\mu+1}\rightarrow 0$
as $\mu\rightarrow 0$, and so the lemma follows. 
\end{proof}

\section{Julia sets for quadratic polynomials - Proof of Theorem \ref{thm3}}

Fix a quadratic polynomial $f_c(z)=z^2+c$ and assume that $c\notin\mathcal M$, where $\mathcal M$ denotes 
the Mandelbrot set.  Then the critical point $0$ is in the basin of attraction of infinity $\Omega_\infty$, 
and the Julia set $\mathcal J_c=\mathbb P^1\setminus\Omega_\infty$ is a Cantor set.   We let 
$G_c(z)$ denote the negative Green's function associated to $f_c$.   It satisfies the following properties:
\begin{itemize}
\item[(1)] $G_c$ is continuous on $\mathbb C$ and harmonic on $\mathbb C\setminus\mathcal J_c$,
\item[(2)] $G_c(z) = -\log |z| + O(1)$ near $\infty$, and
\item[(3)] $G_c(f^n(z)) = 2^{n}G_c(z)$ for all $z\in\mathbb C$.
\end{itemize}
We regard $G_c$ as an exhaustion function of $\Omega_\infty$.  Let $t_0=G_c(0)$.
The exhaustion may be described as follows.  For $t<t_0$ the level 
sets $\Gamma_t=\{G_c=t\}$ are smooth connected embeddings of $S^1$, shrinking around 
infinity as $t$ decreases to $-\infty$.  Considering the picture in $\mathbb C$, as $t$ increases 
to $t_0$, the family $\Gamma_t$ is a decreasing family of embedded $S^1$'s, decreasing 
to $\Gamma_{t_0}$, which is a figure eight, the origin being the figure eight crossing point.  
In general, the level sets $\Gamma_{2^{-n}t_0}$ consists of $2^n$
pairwise disjoint figure eights, and for $2^{-n}t_0<t<2^{-n+1}t_0$ the level set 
$\Gamma_t$ consists of $2^{n+1}$ disjoint smoothly embedded copies of $S^1$, 
one contained in each hole of a figure eight in $\Gamma_{2^{-n}t_0}$.  \

We now assume to get a  contradiction that there exists a bounded holomorphic injection $\varphi:\Omega_\infty\rightarrow\triangle$, 
and we may assume that $\varphi(\infty)=0$.  We will first use the exhaustion just described 
to get a description of $\varphi(\Omega_\infty)$ that will allow us to modify $\varphi$ in a useful way.   Set $H=G_c \circ \varphi^{-1}$, defined 
on $\varphi(\Omega_\infty)$.  \

Start by choosing $s_0<<0$ and let $D_0$ be the disk bounded by $\gamma_{s_0}=\{H=s_0\}$, a single closed loop.  
Increasing $s$ between $s_0$ and $t_0$ we get an increasing family of single loops $\gamma_s$, but when 
$s$ crosses the critical value $t_0$ it breaks into two loops, say $\gamma_{s_1}^1,\gamma_{s_1}^2$, for $s$ close to $t_0$.  
One of these loops is going to enclose the other, and we relabel it $\gamma_{s_1}$.  Next, increasing $s$ between $s_1$ and 
$2t_0$ we follow a path of loops starting from $\gamma_{s_1}$, until $s$ crosses $2t_0$, and it again 
breaks into two loops, say $\gamma_{s_2}^1$ and $\gamma_{s_2}^2$ for $s_2$ close to $2t_0$.  Again, 
single out the one enclosing the other, and relabel it $\gamma_{s_2}$.  Continuing in this fashion, 
we obtain a family of loops $\gamma_{s_j}$ such that $\gamma_{s_j}$ encloses $\gamma_{s_{j-1}}$, and 
such that the disk $D_j$ bounded by $\gamma_{s_j}$ contains the whole sublevel set $\{H<s_j\}$.  We 
have that $\{D_j\}$ is an increasing family of disk,  we denote by $D$ its increasing union, and
we let $\psi:D\rightarrow\triangle$ be the Riemann map satisfying $\psi(0)=0, \psi'(0)>0$.  Our modified 
map will be $\tilde\varphi:=\psi\circ\varphi$.   \ 

Next we will use the map $f_c$ to find some other loops $\tilde\gamma_j$ in $\Omega_\infty$, each 
one in the same free homotopy class as $\varphi^{-1}(\gamma_{s_j})$.  Start by defining $\tilde\gamma_0$
as the level set $G_c=t$ for some $t<t_0$ close to $t_0$.  Then $f_c^{-1}(\tilde\gamma_0)$
consists of two disjoint loops, one of them free homotopic to $\varphi^{-1}(\gamma_{s_1})$.  
Single this out, and label it $\tilde\gamma_1$.  Next $f_c^{-1}(\tilde\gamma_1)$ consists 
of two disjoint loops, and one of them is free homotopic to $\varphi^{-1}(\gamma_{s_2})$. 
Single it out, and denote it by $\tilde\gamma_2$.  Continue in this fashion indefinitely.  \

We are now ready to reach the contradiction.   On the one hand, since the family $\tilde\varphi(\tilde\gamma_j)$
will increase towards $b\triangle$, it follows that the Kobayashi lengths of $\tilde\gamma_j$ in $ \Omega_\infty$
will increase towards infinity.   On the other hand, let $C\subset\Omega_\infty$ denote the forward 
and backward orbit of the critical point $0$.   Then the Kobayashi length of each $\tilde\gamma_j$ in $\Omega_\infty\setminus C$
is longer than the Kobayashi length in $\Omega_\infty$.  But $f_c:\Omega_\infty\setminus C\rightarrow\Omega_\infty\setminus C$
is a covering map, and so the Kobayashi lengths of all the $\tilde\gamma_j$'s in $\Omega_\infty\setminus C$ are the same.  
A contradiction. 

%\section{Some open problems}
%
%
%
%
%\begin{problem}
%For a domain $\Omega\subset\mathbb P^1$, do we have that $R_\Omega=S_\Omega?$
%\end{problem}
%
%
%\begin{problem}
%For various topologies on domains in $\mathbb P^1$, what are the optimal continuity properties of the squeezing function? 
%\end{problem}
%
%For the rest of problems, $K$ is always a Cantor set in $\mathbb P^1$ and $\Omega=\mathbb P^1\setminus K$.
%
%\begin{problem}
%Does there exist $K$ such that 
%\begin{equation}
%0<a<S_\Omega(z)<b<1
%\end{equation}
%for all $z\in\Omega$?  How about if $a=0$?
%\end{problem}
%
%\begin{problem}
%Assume that 
%\begin{equation}
%\lim_{\Omega\ni z\rightarrow K}\frac{g_C(z)}{g_\Omega(z)}\geq a > 0, 
%\end{equation}
%where $g_C$ and $g_K$ denotes the Carath\'{e}odory and Kobayashi metric respectively.   Does it follow that 
%$S_\Omega$ is non-degenerate?  (It does not follow if we replace $g_C$ by the Suita metric.)  If yes, is 
%$\Omega$ uniformly squeezing?  What if $a=1$?
%\end{problem}
%
%\begin{problem}
%Does there exist $K$ with Hausdorff dimension less than two for which $S_\Omega$ is non-degenerate?   How about 
%2-dimensional Hausdorff dimension zero? 
%\end{problem}
%
%\begin{problem}
%Does there exist $K$ with Hausdorff dimension less than two such that $\Omega$ is uniformly squeezing?   How about 
%2-dimensional Hausdorff dimension zero? 
%\end{problem}
%
%

\bibliographystyle{amsplain}

\end{document}